\documentclass[12pt,leqno]{amsart}
\usepackage{amsmath,amssymb,amsfonts}
\usepackage{eucal,enumerate}
\usepackage{xypic,graphicx,psfrag}
\usepackage{mathrsfs}
\usepackage{hyperref}

\usepackage{color} 

\usepackage{tikz}
\usetikzlibrary{decorations.pathreplacing}
\usepackage{scalefnt}

\newcommand\xrefcomment[1]{}		

\newtheorem{lem}{Lemma}[section]
\newtheorem{cor}[lem]{Corollary}
\newtheorem{prop}[lem]{Proposition}
\newtheorem{thm}[lem]{Theorem}

\theoremstyle{definition}
\newtheorem{conj}[lem]{Conjecture}

\newtheorem{problem}{Problem}
\newtheorem{conjecture}[problem]{Conjecture}

\numberwithin{equation}{section}
\numberwithin{table}{section}
\numberwithin{figure}{section}

\allowdisplaybreaks

\newcommand\vstrut[1]{\rule{0ex}{#1}}

\renewcommand\mod{\, \operatorname{mod}\, }
\renewcommand{\phi}{\varphi} 
\renewcommand{\epsilon}{\varepsilon}

\newcommand\codim{\operatorname{codim}}

\newcommand\cA{\mathscr{A}}		
\newcommand\cB{\mathcal{B}}
\newcommand\cBo{{\cB^\circ}}

\renewcommand\cH{\mathcal{H}}	

\renewcommand\cL{\mathscr{L}}	

\newcommand\cT{\mathcal{T}}
\newcommand\cU{\mathcal{U}}		
\newcommand\tcU{\widetilde\cU}

\newcommand\bbR{\mathbb{R}}
\newcommand\bbZ{\mathbb{Z}}

\newcommand\pB{\mathbb B}	
\newcommand\pN{\mathbb N}	
\newcommand\pP{\mathbb P}	
\newcommand\pQ{\mathbb Q}	
\newcommand\pR{\mathbb R}	

\newcommand\bz{\mathbf z}

\newcommand\lcm{\operatorname{lcm}}

\newcommand\M{\mathbf{M}}
\newcommand\Kot{Kot\v{e}\v{s}ovec}
\newcommand\cube{[0,1]^{2q}}
\newcommand\ocube{(0,1)^{2q}}

\newcommand\barn{{\bar n}}
\newcommand\hatc{{\hat c}}
\newcommand\hatd{{\hat d}}

\allowdisplaybreaks

\begin{document}

\renewcommand\langle{(}
\renewcommand\rangle{)}

\newcommand\CQhone{III.3.2\xrefcomment{C:Qh1}}		

\newcommand\Phvdiag{III.3.1\xrefcomment{P:hvdiag}}	
\newcommand\Lcodimzot{III.3.4\xrefcomment{L:codim012}}	

\newcommand\TtwoQhk{III.4.1\xrefcomment{T:2Qhk}}	
\newcommand\TthreeQhk{III.4.2\xrefcomment{T:3Qhk}}	
\newcommand\Tbishopsperiod{VI.1.1\xrefcomment{T:bishopsperiod}}  

\newcommand\Tbpqueensgammatwo{III.3.1\xrefcomment{Tb:pqueensgamma2}}	
\newcommand\Tbpqueensgammathree{III.3.2\xrefcomment{Tb:pqueensgamma3}}	
\newcommand\Tbpqueenspolystwo{III.4.1\xrefcomment{Tb:pqueenspolys2}} 
\newcommand\Tbpqueenspolysthree{III.4.2\xrefcomment{Tb:pqueenspolys3}} 
\newcommand\Tbtypesthree{III.4.3\xrefcomment{Tb:types3}}	

\newcommand\fourmove{IV.6\xrefcomment{IV.sec:4move}}
\newcommand\semiqueen{IV.5.8\xrefcomment{IV.conj:semiqueen}}
\newcommand\onemove{IV.3.1\xrefcomment{IV.p:1move}}
\newcommand\trident{IV.5.10\xrefcomment{IV.conj:trident}}
\newcommand\denomtwomove{IV.4.2\xrefcomment{IV.P:denom2move}}
\newcommand\denomone{IV.2.3\xrefcomment{IV.prop:denom1}}



\title{A $q$-Queens Problem \\
V.  Some of Our Favorite Pieces:  \\Queens, Bishops, Rooks, and Nightriders}

\author{Seth Chaiken}
\address{Computer Science Department\\ The University at Albany (SUNY)\\ Albany, NY 12222, U.S.A.}
\email{\tt sdc@cs.albany.edu}

\author{Christopher R.\ H.\ Hanusa}
\address{Department of Mathematics \\ Queens College (CUNY) \\ 65-30 Kissena Blvd. \\ Queens, NY 11367-1597, U.S.A.}
\email{\tt chanusa@qc.cuny.edu}

\author{Thomas Zaslavsky}
\address{Department of Mathematical Sciences\\ Binghamton University (SUNY)\\ Binghamton, NY 13902-6000, U.S.A.}
\email{\tt zaslav@math.binghamton.edu}

\begin{abstract}
Parts~I--IV showed that the number of ways to place $q$ nonattacking queens or similar chess pieces on an $n\times n$ chessboard is a quasipolynomial function of $n$ whose coefficients are essentially polynomials in $q$.  For partial queens, which have a subset of the queen's moves, we proved complete formulas for these counting quasipolynomials for small numbers of pieces and other formulas for high-order coefficients of the general counting quasipolynomials.  
We found some upper and lower bounds for the periods of those quasipolynomials by calculating explicit denominators of vertices of the inside-out polytope.

Here we discover more about the counting quasipolynomials for partial queens, both familiar and strange, and the nightrider and its subpieces, and we compare our results to the empirical formulas found by \Kot.  We prove some of \Kot's formulas and conjectures about the quasipolynomials and their high-order coefficients, and in some instances go beyond them.
\end{abstract}



\thanks{The first author gratefully acknowledges support from PSC-CUNY Research Awards PSCOOC-40-124, PSCREG-41-303, TRADA-42-115, TRADA-43-127, and TRADA-44-168.  The latter two authors thank the very hospitable Isaac Newton Institute for facilitating their work on this project.}

\maketitle
\pagestyle{myheadings}
\markright{\textsc{A $q$-Queens Problem. V. Favorite Pieces}}\markleft{\textsc{Chaiken, Hanusa, and Zaslavsky}}


\section{Introduction}\label{intro}

We apply our study of nonattacking chess pieces from Parts I--IV \cite{QQs1,QQs2,QQs3,QQs4}  to standard chess pieces---the queen, bishop, and rook (even the rook is interesting)---and our favorite fairy chess piece, the nightrider, which moves any distance in the directions of a knight's move.  And to the  ``partial queens'' and ``partial nightriders''  that, like the bishop and rook, have a subset of the queen's or the nightrider's moves.  

The classic $n$-Queens Problem asks for the number of nonattacking configurations of $n$ queens on an $n\times n$\label{d:n} chessboard.  It has no practical general solution; the only known general formulas are the computationally impractical ones in Part~II \cite{QQs2} and in \cite{Pratt}.
We began our study by separating $n$, the board size, from $q$,\label{d:q} the number of queens, arriving at the \emph{$q$-Queens Problem}: In how many ways can $q$ queens be placed on an $n\times n$ chessboard so that no queen attacks another?  
In Parts I--IV we developed a general geometrical theory of this number, $u_\pQ(q;n)$, as a function of integers $n$ and $q>0$, for all chess and fairy chess pieces that, like the queen, bishop, rook, and nightrider, have moves of unlimited length.  Such pieces are called \emph{riders} by the fairy chess community.\footnote{Fairy chess is chess with unusual pieces, rules, or boards.}  
The problem, given a rider $\pP$\label{d:P}, is this:

\begin{problem}\label{Pr:formula}
Find an explicit formula for $u_\pP(q;n)$\label{d:indistattacks}, the number of nonattacking configurations of $q$ unlabelled pieces $\pP$ on an $n \times n$ board.
\end{problem}

Finding a single comprehensive formula for $u_\pP(q;n)$ for a single piece $\pP$, for all $q$ and $n$, for any piece other than the rook and bishop---and especially for the queen---looks impossible.  
Nevertheless, in Part I we established that $u_\pP(q;n)$ is a quasipolynomial function of $n$ of degree $2q$, that is, it is given by a cyclically repeating sequence of polynomials (the quasipolynomial's \emph{constituents}); so it can be written as  
\begin{equation*}
u_\pP(q;n) = \gamma_0(q) n^{2q} + \gamma_1(q) n^{2q-1} + \gamma_2(q) n^{2q-2} + \cdots + \gamma_{2q}(q) n^0
\label{d:gamma}
\end{equation*}
where the coefficients $\gamma_i(q)$\label{d:coeff} vary cyclically, depending on $n$ modulo a number $p$,\label{d:p} the \emph{period}, but not on $n$ itself.  
The period is a fundamental number; it tells us how much data is needed to rigorously determine the complete quasipolynomial, since $2qp$ values of the counting function suffice.  
Furthermore, in each residue class of $n$, $q!\gamma_i(q)$ is a polynomial function of $q$ of degree $2i$.  Still further, substituting $n=-1$ gives the number of combinatorially distinct types of nonattacking configurations, as explained in Section~\ref{essentials}.  

In this part we present our current state of knowledge about naturally interesting pieces.  The goal is to prove exact formulas for a fixed number $q$ of each piece, where $q$ is (unavoidably) small, and along the way to see how many complete formulas we can prove, for all $q$, for coefficients of high powers of $n$.  
Our contributions include rigorous partial answers to Problem~\ref{Pr:formula} for the queen, bishop, nightrider, and pieces with subsets of their moves.  Some of our rigorous answers, especially for partial nightriders, are new; others were previously known heuristically (by \Kot\ and others; see \cite{ChMath}) but only a few were proved.

Three pieces were previously solved.  The rook and semirook are elementary.  Arshon and \Kot\ proved a complete formula for the bishop, but it is not a quasipolynomial and cannot have $n=-1$ substituted to get the number of combinatorial types; thus we consider the bishops problem only partially solved.

How large a number $q$ and how many coefficients we can handle depend on the piece.  For the rook, naturally, we get well-known formulas for all $q$, although our viewpoint does lead to an apparently new property of Stirling numbers (Proposition~\ref{P:stirling}).  At the other extreme we have only very partially solved three nightriders, for which the formula was previously found heuristically, without proof, by \Kot.  No formula for four nightriders has even been guessed; it is conceivable that it could be obtained by a painstaking analysis using our method, but that is not probable, judging by the 11-digit denominator (see Table~\ref{Tb:nightriders}; the denominator is a multiple of the period, conjecturally equal to it).  We offer no hope for five.

Although in Ehrhart theory periods often are less than denominators, we observe equality in our solved chess problems.  We propose:

\begin{conjecture}\label{Cj:p=D}
For every rider $\pP$ and every $q\geq1$, the period of the counting quasipolynomial $q!u_\pP(q;n)$ equals the denominator of the inside-out polytope for $q$ copies of $\pP$.
\end{conjecture}

We remind the reader that \Kot's book \cite{ChMath} is replete with formulas, mostly generated by himself, for all kinds of nonattacking chess problems.  Properties of \Kot's bishops and queens formulas inspired many of our detailed results.  

Anyone who wants to know the actual number of nonattacking placements of $q$ of our four principal pieces will find answers in the Online Encyclopedia of Integer Sequences \cite{OEIS}.  Table~\ref{Tb:oeis} gives sequence numbers in the OEIS.  The first row is the sequence of square numbers.  After that it gets interesting.  
\begin{table}[htbp]
\begin{center}
\begin{tabular}{|r|c|c|c|c|}
\hline
\vstrut{12pt} $q$	& Rooks		& Bishops		& Queens		& Nightriders \\
\hline
\vstrut{12pt} 1	& A000290	& A000290	& A000290	& A000290 \\
\vstrut{12pt} 2	&\ A163102*	& A172123	& A036464	& A172141 \\
\vstrut{12pt} 3	& A179058	& A172124	& A047659	& A173429 \\
\vstrut{12pt} 4	& A179059	& A172127	& A061994	& ---	\\
\vstrut{12pt} 5	& A179060	& A172129	& A108792	& ---	\\
\vstrut{12pt} 6	& A179061	& A176886	& A176186	& ---	\\
\vstrut{12pt} 7	& A179062	& A187239	& A178721	& ---	\\
\vstrut{12pt} 8	& A179063	& A187240	& ---	& ---	\\
\vstrut{12pt} 9	& A179064	& A187241	& ---	& ---	\\
\vstrut{12pt} 10	& A179065	& A187242	& ---	& ---	\\
\hline
\end{tabular}
\end{center}
\medskip
\caption{Sequence numbers in the OEIS for  nonattacking placements of $q$ rooks, bishops, queens, and nightriders.  In each sequence the board size $n$ varies from 1 to (usually) 1000.  * means $n$ in the OEIS is offset from our value.}
\label{Tb:oeis}
\end{table}

A summary of this paper:  Sections~\ref{essentials} and \ref{ehrhart} recall essential notation and formulas from Parts~I--IV and describe the concepts we use to analyze periods; in Section \ref{sec:parity} we strengthen the Parity Theorem from Part II.   
After connecting to our theory the known results on rooks and semirooks in Section~\ref{R}, we discuss the current state of knowledge and ignorance about bishops and semibishops in Section~\ref{B}.  
Section~\ref{Q} treats the queen as well as the partial queens that are not the rook, semirook, bishop, and semibishop.  Section~\ref{N} concerns the nightrider and its sub-pieces, whose nonattacking placements have not been the topic of any previous theoretical discussion that we are aware of and consequently receive detailed treatment.

We conclude with questions related to these ideas and with proposals for research.  For example, we suggest in Section~\ref{simplified} fairy chess pieces with relatively simple behavior that might provide insight into the central open problem of a good general bound on the period of the counting quasipolynomial in terms of $q$ and the set of moves.  

We append a dictionary of notation for the benefit of the authors and readers.

\section{Essentials, Mostly from Before}\label{essentials}

Our \emph{board} $\cB$\label{d:cB} is the unit square $[0,1]^2$.  (In Section~\ref{semibishop} it may also be a right triangle; for simplicity here we assume the square.) 
Pieces are placed on integral points, $(x,y)$ for $x,y \in [n] := \{1,\ldots,n\}$, in the interior of an integral dilation $(n+1)\cB$ of the board; e.g., the square board $(n+1)[0,1]^2$ has interior $(n+1)(0,1)^2$.  We also call these integral points the \emph{board}; it will always be clear which board we mean.  

A piece $\pP$ has \emph{moves}, which are the integral multiples of a finite set $\M$\label{d:moveset} of \emph{basic moves}, which are non-zero, non-parallel integral vectors $m = (c,d) \in \bbZ^2$\label{d:mr}.  
For instance, $\M=\{(1,0),(1,1),(0,1), (1,-1)\}$ for the queen and $\M=\{(2,1),(1,2),(2,-1),(1,-2)\}$ for the nightrider.  
Each basic move must be reduced to lowest terms and no basic move may be a scalar multiple of another.  Thus, the slope $d/c$\label{d:slope-hyp} contains all  necessary information and can be specified instead of $m$ itself.  
We say two distinct pieces \emph{attack} each other if the difference of their locations is a move.  
In other words, a piece in position $z:=(x,y) \in \bbZ^2$ attacks any other piece in the lines $z + r m$ for $r \in \bbZ$ and $m \in \M$.  
Attacks are not blocked by a piece in between, and they include the case where two pieces occupy the same location.

A \emph{configuration} $\bz=(z_1,\ldots,z_q)$\label{d:config} is a choice of locations $z_i:=(x_i,y_i)$ for the $q$ pieces on the dilated board $(n+1)\cB$, including the boundary.  Thus $\bz$ is an integral point in the dilated polytope $(n+1)\cB^q = (n+1)\cube$.\label{d:cP}  
For a \emph{nonattacking configuration} the pieces must be in the interior, $(n+1)(0,1)^2$, and no two pieces may attack each other.  In other words, if there are pieces at positions $z_i$ and $z_j$, then $z_j-z_i$ is not a multiple of any $m\in\M$.  

The attack lines determine \emph{move hyperplanes} in $\bbR^{2q}$,\label{d:configsp} $\cH_{ij}^{d/c}$: $(z_j-z_i)\cdot m^\perp = 0$ for each $m=(c,d)\in\M$, where $m^\perp := (d,-c)$,\label{d:mrperp} whose integral points in the dilated open hypercube $(n+1)\ocube$ represent attacking configurations.  The nonattacking configurations are the integral points in the open hypercube and outside every hyperplane, so it is they we want to count.  
The combination of the hypercube $\cube$ and the set $\cA_\pP$\label{d:AP} of hyperplanes is called an \emph{inside-out polytope} \cite{IOP}.  Its \emph{vertices} are all the points in $\cube$ that are intersections of move hyperplanes and facet hyperplanes of $\cube$.  
From \cite{IOP} we know $o_\pP(q;n)$ is a quasipolynomial whose period divides the \emph{denominator} $D(\cube,\cA_\pP)$,\label{d:D} defined as the least common multiple of the denominators of all coordinates of vertices.  
In Section \ref{Nprep} we continue this exposition for use with the nightrider.

The bishop, rook, and queen are examples of \emph{partial queens}.  
A partial queen $\pQ^{hk}$\label{d:partQ} is a rider that has $h$ horizontal or vertical basic moves and $k$ basic moves at $\pm45^\circ$ to the horizontal (obviously, $h,k\leq2$).  We studied partial queens in Part~III.  
Table~\ref{Tb:pqueensperiods} contains a list of the partial queens with their names and what we know or believe about the periods of their counting quasipolynomials.  
We think four of the partial queens are uniquely special: we believe the rook, semirook, semibishop, and anassa are the only rider pieces that have period 1---that is, whose labelled counting functions are polynomials in $n$ for every number $q$ of copies.  We do know from Theorem \denomone\ that they are the only riders whose denominators equal 1.  

\begin{table}[hbt]
\begin{center}
\begin{tabular}{|l|l||c|c|c|c|c|c|}
\hline
\quad \emph{Name} & $(h,k)$ \vstrut{15pt} &$q=2$	& $q=3$	& $q=4$	& $q=5$	& $q=6$	& $q>6$	
\\[2pt] \hline
Semirook & $(1,0)$ \vstrut{13pt}
&1	&1	&1	&1	&1	&1
\\[2pt] \hline
Rook & $(2,0)$ \vstrut{13pt}
&1	&1	&1	&1	&1	&1
\\[2pt] \hline
Semibishop & $(0,1)$ \vstrut{13pt}
&1	&1	&1	&1	&1	&1
\\[2pt] \hline
Anassa	& $(1,1)$ \vstrut{13pt}
&1	&1	&1	&1	&1	&1* ($q=7,8$)
\\[2pt] \hline
Bishop & $(0,2)$ \vstrut{13pt}
&1	&2	&2	&2	&2	&2
\\[2pt] \hline
Semiqueen & $(2,1)$ \vstrut{13pt}
&1	&1	&2*	&2*	&6*	&12$^?$ ($q=7$)	
\\[2pt] \hline
Trident 	& $(1,2)$ \vstrut{13pt}
&1	&2	&6*	&12$^?$	&60$^?$	&420$^?$ ($q=7$)
\\[2pt] \hline
Queen & $(2,2)$ \vstrut{13pt}
&1	&2	&6* (6)	&60*	&$840^{**}$	&360360* ($q=7$)	
\\ [2 pt] \hline
\end{tabular}
\bigskip
\caption{The quasipolynomial periods (and denominators, when known) for partial queens $\pQ^{hk}$.   Denominators are defined in Section~\ref{Nprep}.  All known periods equal the denominators.
{\small 
\newline* is a number deduced from an empirical formula in \cite{ChMath}.  
\newline** is deduced from the empirical formula of Karavaev; see \cite{KaravaevWeb, ChMath}.
\newline${}^?$ is a value we conjecture.
} 
}
\label{Tb:pqueensperiods}
\end{center}
\end{table}

While the number of nonattacking configurations of unlabelled pieces is $u_\pP(q;n)$, the counting is done with labelled pieces; that number is $o_\pP(q;n) = q!u_\pP(q;n)$\label{d:distattacks}.  
Our task is to find the coefficients $\gamma_i(q)$, or in practice $q! \gamma_i(q)$, which we know to be polynomials in $q$ that may differ for each residue class of $n$ modulo the period $p$ (Theorem~I.4.2).  Ehrhart theory says that the leading coefficient of $o_\pP(q;n)$ is the volume of the polytope $\cube$, i.e., 1; so $\gamma_0=1/q!$ for every piece.  

Two nonattacking configurations of labelled pieces are said to have the same \emph{labelled combinatorial type} if for each two pieces $\pP_i$ and $\pP_j$ and move $m_k$, $\pP_j$ lies on the same side of the move line through $\pP_i$ in the direction of $m_k$ in both configurations.  That is, $(z_j-z_i)\cdot m^\perp$ should have the same sign in both.  (Section I.5 explains more about combinatorial types.)  Two unlabelled configurations have the same \emph{combinatorial type} if they can be labelled to have the same labelled combinatorial type.  For instance, the number of combinatorial types of one piece is 1, and the number of types for 2 copies of a piece with $r$ moves is $r$.  We proved in Theorem I.5.3 that the number of combinatorial types of nonattacking configuration of $q$ copies of a piece $\pP$ is $u_\pP(q;-1)$.  Table~\ref{Tb:pqueenstypes} shows the number of combinatorial types for small numbers of partial queens.

\begin{table}[hbt]
\begin{center}
\begin{tabular}{|l|l||c|c|c|c|c|c|}
\hline
\quad \emph{Name} & $(h,k)$ \vstrut{15pt} &$q=2$	& $q=3$	& $q=4$	& $q=5$	& $q=6$& $q>6$	
\\[2pt] \hline \hline
Semirook & $(1,0)$ \vstrut{13pt}
&1	&1	&1	&1	&1	&1	
\\[2pt] \hline
Rook & $(2,0)$ \vstrut{13pt}
&2	&6	&24	&120	&720	&$q!$	
\\[2pt] \hline
Semibishop & $(0,1)$ \vstrut{13pt}
&1	&1	&1	&1	&1	&1	
\\[2pt] \hline
Anassa	& $(1,1)$ \vstrut{13pt}
&2	&6	&24	&120	&720	&$q!$	
\\[2pt] \hline
Bishop & $(0,2)$ \vstrut{13pt}
&2	&6	&24	&120	&720	&$q!$	
\\[2pt] \hline
Semiqueen & $(2,1)$ \vstrut{13pt}
&3	&17	&	&	&	&	
\\[2pt] \hline
Trident 	& $(1,2)$ \vstrut{13pt}
&3	&17	&	&	&	&	
\\[2pt] \hline
Queen & $(2,2)$ \vstrut{13pt}
&4	&36	&574*	&14206*	&501552**	&	
\\ [2 pt] \hline
\end{tabular}
\bigskip
\caption{The number of combinatorial types of configuration for nonattacking partial queens $\pQ^{hk}$.  Values for $h+k\leq2$ are from Theorem I.5.8.  Others are from Table III.4.3.
{\small 
\newline* is a number inferred from an empirical formula in \cite{ChMath}.  
\newline** is inferred from the empirical formula of Karavaev; see \cite{KaravaevWeb, ChMath}.
} 
}
\label{Tb:pqueenstypes}
\end{center}
\end{table}
%

\section{The Rook and Its Squire }\label{R}\


\subsection{The rook }\label{rook}\

Rooks illustrate our approach nicely because they are well understood.  
The basic move set, of course, is $\M_\pR = \{(1,0),(0,1)\}$.  
The well-known elementary formula is 
\begin{equation}\label{E:rooks}
u_\pR(q;n) = q! \binom{n}{q}^2 = \frac{1}{q!} (n)_q^2,  
\end{equation}
where $(n)_q$ denotes the falling factorial.  Thus, $o_\pR(q;n)=q!u_\pR(q;n)$ is a quasipolynomial of period 1 (that is, a polynomial) and degree $2q$, in accordance with our general theory.  
In our approach, we want to study its coefficients.

The coefficient of $n^{2q-i}$ is
\begin{equation}
q! \gamma_i = \sum_{k=0}^i s(q,q-k)s(q,q-(i-k)), 
\label{E:rookcoeffs}
\end{equation}
where $s(q,j)$\label{d:s} denotes the Stirling number of the first kind, defined as $0$ if $j<0$ or $j>q$.  
For instance,
\begin{align*}
q!\gamma_0 &= 1,  
&q!\gamma_1 &= -(q)_2,  
\\
q!\gamma_2 &= (q)_2 \frac{3q^2-5q+1}{6},  
&q!\gamma_3 &= -(q)_3 \frac{q(q-1)^2}{6}.
\end{align*}
These formulas, derived from Equation \eqref{E:rookcoeffs}, agree with the general partial queens formulas in Theorem~\Phvdiag.  (Recall that the rook is the partial queen $\pQ^{20}$.)

The sign of each term in the summations in \eqref{E:rookcoeffs} is $(-1)^i$, so that is the sign of $\gamma_i$ for $0\leq i \leq 2q-2$.  For $i>2q-2$, $\gamma_i=0$ because $s(q,0)=0.$  
The rook is one of few pieces for which we know the sign of every term in $o_\pP(q;n)$.

\begin{prop}\label{P:rookform}
The coefficient $q!\gamma_i$ is a polynomial in $q$ of degree $2i$.  It has a factor $(q)_{\lceil i/2 \rceil+1}$.  
The coefficient of $q^{2i}$ is 
\begin{equation}
\frac{1}{(2i)!} \sum_{k=0}^i \binom{2i}{2k} 
\sum_{r=0}^k (-1)^r \binom{2k}{k+r} S(k+r,r) 
\sum_{s=0}^{i-k} (-1)^s \binom{2(i-k)}{i-k+s} S(i-k+s,s) ,
\label{E:rookleading}
\end{equation}
whose sign is $(-1)^i$.
\end{prop}

\begin{proof}
Schl\"omilch's formula \cite[p.\ 216]{Comtet}
\begin{equation}
s(q,q-k) = \sum_{r=0}^k (-1)^r \binom{q-1+r}{k+r} \binom{q+k}{k-r} S(k+r,r) 
\label{E:schlomilch}
\end{equation}
(which involves the Stirling numbers $S(n,k)$\label{d:S} of the second kind) tells us that $s(q,q-k)$ is a polynomial in $q$ of degree $2k$ with leading term
\begin{align*}
\sum_{r=0}^k (-1)^r \frac{q^{k+r}}{(k+r)!} \frac{q^{k-r}}{(k-r)!} S(k+r,r) 
&= \frac{q^{2k}}{(2k)!} \sum_{r=0}^k (-1)^r \binom{2k}{k+r} S(k+r,r)  .
\end{align*}
This term, as the leading term, must have the same sign as $s(q,q-k)$ for large $q$. 
So the leading coefficient of $q!\gamma_i$ is as in Equation~\eqref{E:rookleading} and the sign of this coefficient is $(-1)^i$.  

It is easy to infer from \eqref{E:rookcoeffs} that the polynomial equals 0 if $i > 2q-2$, i.e., $q \leq \lceil i/2 \rceil$; therefore $q(q-1)\cdots(q-\lceil i/2 \rceil)$ is a factor.
\end{proof}

The number of combinatorial types of nonattacking configuration of $q$ (unlabelled) rooks is $q!$.  To prove it we may substitute $n=-1$ into Equation~\eqref{E:rooks} (by Theorem~I.5.3) or apply Theorem~I.5.8, which says that every piece with two basic moves has $q!$ combinatorial configuration types.


\subsection{The semirook}\label{semirook}\

The semirook $\pQ^{10}$ has only one of the rook's moves and is consequently the least interesting of all riders, except that it is the second of the four known pieces with period 1.  Also, because it has no diagonal move, it exemplifies Corollary~\CQhone, that partial queens with at most one diagonal move have a coefficient $\gamma_5$ that is independent of $n$.  
The easy counting formula is
$$u_{\pQ^{10}}(q;n) = \binom{n}{q} n^q = \frac{1}{q!} \sum_{i=0}^q s(q,q-i) n^{2q-i},$$
agreeing with Proposition~II.6.1, where one should take $(c,d)=(1,0)$, and with Tables~\Tbpqueenspolystwo\ and \Tbpqueenspolysthree\ for $q=2, 3$.  The leading coefficients of those polynomials are stated in Theorem~\Phvdiag\ and Tables~\Tbpqueensgammatwo\ and \Tbpqueensgammathree.  Schl\"omilch's formula \eqref{E:schlomilch} shows that, as our theory says, $q!\gamma_i = s(q,q-i)$ is a polynomial in $q$ of degree $2i$ for each $i=1,\ldots,q$.


\section{The Bishop and Its Scion}\label{B}\

Here we treat the bishop and its scion the semibishop.


\subsection{The bishop }\label{bishop}\

The basic move set is $\M_\pB=\{(1,1),(1,-1)\}$.  
The quasipolynomial formulas for up to 6 bishops, published by \Kot\ in early editions of \cite{ChMath}---most of which were found by him---are:  
\begin{equation}\label{E:Bkot}
\begin{aligned}
&u_\pB(1;n) = n^2. \\[5pt]
&u_\pB(2;n) = \frac{n^4}{2}-\frac{2 n^3}{3}+\frac{n^2}{2}-\frac{n}{3}. \\[5pt]
&u_\pB(3;n) = \left\{\frac{n^6}{6}-\frac{2 n^5}{3}+\frac{5 n^4}{4}-\frac{5 n^3}{3}+\frac{4 n^2}{3}-\frac{2 n}{3}+\frac{1}{8}\right\} - (-1)^n\dfrac{1}{8}. \\[5pt]
&u_\pB(4;n) = 
\left\{\frac{n^8}{24}-\frac{n^7}{3}+\frac{11 n^6}{9}-\frac{29 n^5}{10}+\frac{355 n^4}{72}-\frac{35 n^3}{6}+\frac{337 n^2}{72}-\frac{73 n}{30}+\frac{1}{2}\right\}  \\
&\qquad\qquad\ - (-1)^n\left\{\frac{n^2}{8}-\frac{n}{2}+\frac{1}{2}\right\}.\\[5pt]
&u_\pB(5;n) = \left\{\frac{n^{10}}{120}-\frac{n^9}{9}+\frac{49 n^8}{72}-\frac{118 n^7}{45}+\frac{523 n^6}{72}-\frac{2731 n^5}{180}+\frac{3413 n^4}{144}-\frac{4853 n^3}{180} \right. \\
&\qquad\qquad\quad \left.+\frac{2599 n^2}{120}-\frac{1321 n}{120}+\frac{9}{4}\right\}  - (-1)^n\left\{\frac{n^4}{16}-\frac{7 n^3}{12}+\frac{17 n^2}{8}-\frac{85 n}{24}+\frac{9}{4}\right\}.\\[5pt]
&u_\pB(6;n) = \left\{\frac{n^{12}}{720}-\frac{n^{11}}{36}+\frac{37n^{10}}{144}-\frac{4813n^9}{3240}+\frac{8819n^8}{1440}-\frac{72991n^7}{3780}+\frac{2873n^6}{60} \right. & \\
&\qquad\qquad\quad \left. -\frac{100459n^5}{1080}+\frac{199519n^4}{1440}-\frac{498557n^3}{3240}+\frac{14579n^2}{120}-\frac{7517n}{126}+\frac{765}{64}\right\} & \\
&\qquad\qquad\ - (-1)^n\left\{\frac{n^6}{48}-\frac{n^5}{3}+\frac{221n^4}{96}-\frac{211n^3}{24}+\frac{467n^2}{24}-\frac{47n}{2}+\frac{765}{64}\right\}.
\end{aligned}
\end{equation}
For $q\leq4$ these were rigorously proved by Dudeney and Fabel (see \Kot\ \cite[p.\ 234]{ChMath} for these attributions and citations).  The formulas for $q=2,3$ are special cases of our Theorems~\TtwoQhk\ and \TthreeQhk, thereby reinforcing the correctness of those theorems.  
\Kot\ found the formulas for $q=5,6$ heuristically, by calculating the values $u_\pB(q;n)$ for many values of $n$, looking for an empirical recurrence relation, deducing a generating function, and from that getting the quasipolynomial.  
(See Section~\ref{rr} for more about his method.)  
His approach, while excellent for finding formulas, does not prove their validity because it does not bound the period---even though period 2 is plausible since one could guess that odd and even board sizes would have separate polynomials.  

We have proved that 2 is the complete story on the period.  
In Theorem~\Tbishopsperiod\ \cite{QQs6} we provide the missing upper bound that rigorously establishes period 2 for every $q>2$ and hence the correctness of \Kot's quasipolynomial formulas.\footnote{Stanley in \cite[Solution to Exercise 4.42]{EC1-2} says that both quasipolynomiality and the period follow directly from Arshon's formulas; however, we believe such a derivation would be difficult.}  
Together with the fact that we know the degree $2q$ and the leading coefficient $1/q!$ of the constituent polynomials, this implies that, if the first $4q$ values of a candidate quasipolynomial are correct, then we have $u_\pB(q;n)$.  Since \Kot\ did check those values for $q \leq 6$ \cite{Kotpc}, his formulas are proved.

\begin{thm}\label{T:B}
All the formulas in Equations \eqref{E:Bkot} are correct.
\end{thm}

Despite the overall period 2, in \Kot's formulas \eqref{E:Bkot} the six leading coefficients do not vary with the parity of $n$.  \Kot\ conjectured expressions for $\gamma_{1}$, $\gamma_{2}$, and $\gamma_{3}$ in terms of $q$ alone; we proved them in Theorem~\Phvdiag\ (and see Tables~\Tbpqueensgammatwo\ and \Tbpqueensgammathree) since the bishop is the partial queen $\pQ^{02}$.  That theorem also gives the periods of $\gamma_4$ and $\gamma_5$.

\begin{cor}[of Theorem~\Phvdiag]\label{C:Btopcoeffs}
In $u_\pB(q;n)$ the coefficients $\gamma_i$ for $i\leq5$ are constant as functions of $n$.  
\end{cor}

As the number of combinatorial types of nonattacking configuration of $q$ unlabelled bishops is $q!$ by Theorem~I.5.8, we know this is the value of $u_\pB(q;-1)$ even though we do not know the general formula for $u_\pB(q;n)$.

\medskip
A surprising development during our work on this project was \Kot's discovery that, in 1936, Arshon had solved the $n$-bishops problem, the number of ways to place $n$ nonattacking bishops on an $n\times n$ board \cite{Arshon}.  His method was to count independently the number of ways to place $i$ nonattacking bishops on the black squares and on the white squares. 
This work was forgotten until \Kot\ rediscovered it.  It was an easy step for him to write down an explicit formula for the number of placements of $q$ nonattacking bishops \cite[fourth ed., p.\ 140]{ChMath}.  
\Kot\ then restated the Arshon equations with no subtractive terms by using Stirling numbers of the second kind \cite[fourth ed., p.\ 142]{ChMath}.  His formula for $q$ bishops is
\begin{equation}
\begin{aligned}
u_\pB(q;n) = &\sum_{i=0}^q 
\sum_{j=0}^{\left\lfloor\tfrac{n+1}{2}\right\rfloor} \binom{\lfloor\tfrac{n+1}{2}\rfloor}{j} S\big(j+\lfloor \tfrac{n}{2} \rfloor, n-i\big) 
\cdot 
\sum_{h=0}^{\left\lfloor \tfrac{n}{2} \right\rfloor} \binom{\lfloor \tfrac{n}{2} \rfloor}{h} S\big(h+\lfloor \tfrac{n+1}{2} \rfloor, n-(q-i)\big) .
\label{E:AKbish}
\end{aligned}
\end{equation}

For us this is not entirely satisfactory.  Since the number of terms depends on $n$, \eqref{E:AKbish} does not give the quasipolynomial form of $u_\pB(q;n)$ and does not allow us to substitute $n=-1$ to obtain the number of combinatorial types of nonattacking configuration (though for the bishop this number is known, obtained from Theorem~I.5.8).  
We consequently take the point of view that bishops formulas, like those for other pieces, call for a quasipolynomial analysis via Ehrhart theory, so there is room for further work.

\subsection{The semibishop}\label{semibishop}\

We now come to the third piece known to have period 1.  
The \emph{semibishop} $\pQ^{01}$ has just one of the bishop's moves, say $(c,d)=(1,1)$.  Thus, it is an example of a one-move rider (Section II.6).  As such it has counting functions
\begin{align*}
u_{\pQ^{01}}(1;n) &= n^2, 
\\
u_{\pQ^{01}}(2;n) &=  \frac{1}{2} n^4 - \frac{1}{3} n^3 - \frac{1}{6} n, 
\\
u_{\pQ^{01}}(3;n) &=
\frac{1}{6} n^6 - \frac{1}{3} n^5 + \frac{1}{6} n^4 - \frac{1}{6} n^3 + \frac{1}{6} n^2  ,
\\
u_{\pQ^{01}}(4;n) &=
\frac{1}{24} n^8 - \frac{1}{6}  n^7 + \frac{2}{9} n^6 - \frac{11}{60} n^5 
+ \frac{2}{9} n^4 - \frac{1}{6}  n^3 + \frac{1}{72} n^2+\frac{1}{60} n,
\end{align*}
from Proposition~II.6.1 (in which $\barn=0$).  All of these are polynomials in $n$.  

\begin{thm}\label{P:semibishop}
The counting function for nonattacking unlabelled semibishops on the square board is  
$$
u_{\pQ^{01}}(q;n) = (-1)^q \sum_{k=0}^q s(n+1,n+1-k) s(n,n-(q-k)),
$$
which is a polynomial function of $n$ of degree $2q$.  
\end{thm}

\begin{proof}
This is an immediate consequence of Proposition~\ref{P:triangleboard} below.  Alternatively, it can be proved similarly to that proposition.  
\end{proof}

Explicit formulas for the coefficients $\gamma_i$ for $i\leq3$ are in Theorem~\Phvdiag\ and Tables~\Tbpqueensgammatwo\ and \Tbpqueensgammathree.  

\Kot\ independently proposed this same formula in \cite[fourth ed., p.\ 155; sixth ed., p.\ 265]{ChMath} and verified it for $n\leq20$ and some values of $q$, without a proof.

We prepare for the proof of Theorem~\ref{P:semibishop} by changing the board.  The \emph{right triangle board} (\emph{triangular board} for short) has legs parallel to the axes and hypotenuse in the direction of the semibishop's move; thus, it is the set $\cT := \{(x,y) \in \bbR^2 : 0 \leq x \leq y \leq 1\}$\label{d:cT}.  The \emph{$n\times n$ triangular board} is the set of integral points in the interior of the dilation by $n+2$, i.e., 
$$
(n+2)\cT^\circ \cap \bbZ^2 = \{ (x,y) \in \bbZ^2 : 1 \leq x \leq y-1 \leq n \}.
$$
Write $u^\cT_{\pQ^{01}}(q;n)$ for the counting function of nonattacking placements  of $q$ unlabelled semibishops on an $n\times n$ triangular board.  Most of our theory for the square board applies equally well to the triangular board; we omit details.

\begin{prop}\label{P:triangleboard}
The counting function for nonattacking unlabelled semibishops on the triangular board is  
$u^\cT_{\pQ^{01}}(q;n) = (-1)^q s(n+1,n+1-q).$
\end{prop}

\begin{prop}\label{P:stirling}
The Stirling number of the first kind, $s(n+1,n+1-q)$, is a polynomial function of $n$ of degree $q$.  The coefficient of $n^{2q-i}$ in $q!(-1)^q s(n+1,n+1-q)$ is a polynomial function of $q$ of degree $2i$.
\end{prop}

The fact that $s(n+1,n+1-q)$ is a polynomial in $n$ of degree $2q$ follows from Schl\"omilch's formula \eqref{E:schlomilch} and is well known (see e.g.\ \cite{GS}).  We give a proof here which we believe to be new, using Ehrhart theory in the spirit of our chess series.  
We do not know a prior reference for the fact that the coefficients themselves are polynomials.

\begin{proof}
We prove both propositions together.

The $n\times n$ integral right triangle board has $n$ diagonals (parallel to the hypotenuse) of lengths $1,2,\ldots,n$, each of which can have at most one semibishop.  The number of ways to place $q$ labelled semibishops is the sum of all products of $q$ of these $n$ values, i.e., the elementary symmetric function $e_q(1,2,\ldots,n)$, which equals $|s(n+1,n+1-q)|$.  That proves Proposition \ref{P:triangleboard}.

Next, we prove $s(n+1,n+1-q)$ is a polynomial function of $n$.  
Proposition~\onemove\ applies because the semibishop is a one-move rider.  
The corners of the triangular board $\cT$ are $(0,0)$, $(1,1)$, and $(0,1)$.  For the move $(1,1)$, their antipodes are the opposite points of the boundary along the $45^\circ$ diagonal.  
The corner $(0,1)$ has no antipode while the corners $(0,0)$ and $(1,1)$ serve as each other's antipodes.  Thus, every vertex is integral and the denominator $D(\cT^q,\cA_{\pQ^{01}}) = 1$, so $q! u^\cT_{\pQ^{01}}(q;n)$ is a polynomial in $n$.  

Finally, Theorem~I.4.2 says that the coefficients are polynomials in $q$ with the stated degrees.
\end{proof}


\section{The Queen and Its Less Hardy Sisters}\label{Q}\

There are four pieces in this section: the queen, and three others that are like defective queens without being similar to bishops or rooks, which we call the semiqueen, the trident, and the anassa.  As they are all partial queens, counting formulas for $q=2,3$ are special cases of our Theorems~\TtwoQhk\ and \TthreeQhk\ and are presented in Tables~\Tbpqueenspolystwo\ and \Tbpqueenspolysthree.  The four leading coefficients of the general counting polynomial are implied by Theorems~\Phvdiag, \TthreeQhk, and \TthreeQhk;  see Tables~\Tbpqueensgammatwo\ and \Tbpqueensgammathree.  All our formulas that were also calculated by \Kot\ in \cite{ChMath} agree with his.

Let $\zeta_r := e^{2\pi i/r}$\label{d:zetar} be a primitive $r$-th root of unity.
We index the Fibonacci numbers $F_i$\label{d:Fib} so that $F_0=F_1=1$.


\subsection{The queen}\label{queen}\

The basic move set is $\M_\pQ=\{(1,0),(0,1),(1,1),(1,-1)\}$.  
The known quasipolynomial formulas for up to four queens are:  
\begin{equation}\label{E:Qkot}
\begin{aligned}
u_\pQ(1;n) &= n^2. \\[5pt]
u_\pQ(2;n) &= \frac{n^4}{2}-\frac{5 n^3}{3}+\frac{3 n^2}{2}-\frac{n}{3} = \frac{n(n-1)(3n^2-7n+2)}{6} 
 \\[5pt]
u_\pQ(3;n) &= \left\{\frac{n^6}{6}-\frac{5 n^5}{3}+\frac{79 n^4}{12}-\frac{25 n^3}{2}+11 n^2-\frac{43 n}{12}+\frac{1}{8}\right\} + (-1)^n \left\{\frac{n}{4}-\frac{1}{8}\right\}. \\[5pt]
u_\pQ(4;n) &= \left\{\frac{n^8}{24}-\frac{5n^7}{6}+\frac{65n^6}{9}-\frac{1051n^5}{30}+\frac{817n^4}{8}-\frac{19103n^3}{108}+\frac{3989n^2}{24}-\frac{18131n}{270}+\frac{253}{54}\right\} \\
&\quad + (-1)^n \left\{\frac{n^3}{4}-\frac{21n^2}{8}+7n-\frac{7}{2}\right\} 
 + \operatorname{Re}(\zeta_3^n) \frac{32(n-1)}{27} 
 + \operatorname{Im}(\zeta_3^n) \frac{40\sqrt3}{81}.
\end{aligned}
\end{equation}
The square and cube roots of unity in $u_\pQ(4;n)$ imply period 6.

The formula for two queens, originally due to Lucas, is given by our Theorem~\TtwoQhk.  The formula for three queens, due to Landau, is implied by our Theorem~\TthreeQhk.  \Kot\ gives formulas for up to six queens, calculated by him for $q=4, 5$ and calculated for six queens by Artem M.\ Karavaev.  (\cite{ChMath} has the citations.)  These three have not been not rigorously proved.

\Kot\ conjectured formulas for $\gamma_{1}$ and $\gamma_{2}$ based on the known and heuristically derived formulas (mostly by him) for $u_\pQ(q;n)$ for small $q$.  Theorem~\Phvdiag\ proves his conjectures and also a formula for $\gamma_{3}$, and further proves that $\gamma_{4}$ is constant as a function of $n$, but that the next two coefficients are not.

\begin{cor}[of Theorem~\Phvdiag]\label{C:Qtopcoeffs}
In $u_\pQ(q;n)$ the coefficients $\gamma_i$ for $i\leq4$ are constant as functions of $n$; but $\gamma_5$ has period $2$ if $q\geq3$.  Exact formulas are
\begin{align*}
\gamma_0 &= \frac{1}{q!}, \qquad 
\gamma_1 = -\, \frac{1}{(q-2)!}\bigg\{ \frac{5}{3} \bigg\}
, \\
\gamma_2 &= \frac{1}{2!(q-2)!} \bigg\{ \frac{25}{9} (q-2)_2 + \frac{61}{6} (q-2) + 3 \bigg\} 
, \\
\gamma_3 &= -\, \frac{1}{3!(q-2)!} \bigg\{ \frac{125}{27} (q-2)_4 + \frac{305}{6} (q-2)_3 + \frac{681}{5} (q-2)_2 + 73 (q-2) + 2 \bigg\} .
\end{align*}
\end{cor}

Unlike the case of bishops and semibishops, the period of $u_\pQ(q;n)$ is not simple, although \Kot\ \cite[6th ed., p.\ 31]{ChMath} makes the following remarkable conjecture.

\begin{conj}[\Kot]\label{KotFib}
The counting quasipolynomial for $q$ queens has period $\lcm[F_q]$, the least common multiple of all positive integers up through the $q$-th Fibonacci number $F_q$.
\end{conj}

The observed periods up to $q=7$ (see \cite[6th ed., pp.\ 19, 27--28]{ChMath} for $q=7$) agree with this proposal, and the theory of Section~\fourmove\ lends credence to its veracity.  

\Kot\ conjectures, yet more strongly, the exact form of the denominator of the generating function $\sum_{n\geq0} u_\pQ(q;n) x^n$: it is a product of specific cyclotomic polynomials raised to specific powers; see \cite[third ed., pp.\ 11, 14 {\it et seq.}]{ChMath},  \cite[6th ed., p.\ 22]{ChMath}.
The conjecture implies that, when written in standard Ehrhart form with denominator $(1-x^p)^{2q+1}$, the generating function has many cancellable factors.  This, too, is not predicted by Ehrhart theory; but as it is too systematic and elegant to be accidental, it presents another tantalizing question. \Kot's evidence, indeed, suggests that $u_\pQ(q;n)$ has a recurrence relation of length far less than its period (see Section~\ref{rr}).  A proof of these conjectures seems to call for a new theoretical leap forward.


\subsection{The semiqueen}\label{semiqueen}\

The \emph{semiqueen} $\pQ^{21}$ is the queen without one of its diagonal moves (think of it as having lost the left or right arm in battle).  We gave formulas for $q=2$ and $3$ in Tables III.4.1 and III.4.2.  We also gave formulas for the initial coefficients $\gamma_2$ and $\gamma_3$ in Tables III.3.1 and III.3.2.  For higher values of $q$ we refer to \Kot's heuristic counting formulas for $q\leq 6$ in \cite[6th ed., pp.\ 732--733]{ChMath}.

Conjecture~\semiqueen\ states a conjectural upper bound for the quasipolynomial period of $\lcm[F_{q/2}]$ when $q$ is even and $\lcm[F_{(q+1)/2}-1]$ when $q$ is odd.  Since we do not expect all denominators in $[F_{q/2}]$ or $[F_{(q+1)/2}-1]$ to appear, we do not expect this bound to be tight for large $q$, although it agrees with \Kot's formulas for $q\leq 6$.


\subsection{The trident}\label{sec:trident}\

The \emph{trident} is the partial queen $\pQ^{12}$ that can advance and retreat but cannot move sideways. 
Our counting formulas for $q\leq3$ in Tables III.4.1 and III.4.2 are the same as \Kot's heuristic ones \cite[pp.\ 730--731]{ChMath}, thereby confirming his.   Tables III.3.1 and III.3.2 give the values of $\gamma_2$ and $\gamma_3$. 

Conjecture~\trident\ implies a conjectural upper bound for the period of the counting quasipolynomial for $q$ tridents of $\lcm[2F_{q/2}-1]$ when $q$ is even and $\lcm[F_{(q+3)/2}-1]$ when $q$ is odd.  Again, this should not be considered tight.  However, \Kot's heuristic approach did yield a period of 6 when $q=4$, which is our upper bound.


\subsection{The anassa}\label{anassa}\

The \emph{anassa} (\Kot's ``semi-rook + semi-bishop'') is the partial queen $\pQ^{11}$, with one horizontal or vertical and one diagonal move.  It is the fourth and final piece known to have period 1, i.e., whose counting function $u_{\pQ^{11}}(q;n)$ is a polynomial for all $q$.

\Kot\ noticed in his results that $(n)_q$ is a factor of $u_{\pQ^{11}}(q;n)$ for $q$ up to $8$.  For instance,
\[
u_{\pQ^{11}}(6;n) = \frac{(n)_6}{6!}\bigg(n^6-10n^5+45n^4-\frac{1093}{9}n^3+\frac{634}{3}n^2-\frac{14033}{63}n+\frac{2278}{21}\bigg).
\]
We would like an explanation for this.

\Kot\ also presents a formula for the number of ways to place $n$ nonattacking anassas on an $n\times n$ board: 
\[
u_{\pQ^{11}}(n;n)=\sum_{k=1}^n\binom{n+1}{k}\frac{k!}{2^k}S(n,k).
\]

\section{Preparation for the Nightrider }\label{Nprep}\

\subsection{Ehrhart enumeration}\label{ehrhart}\

For the nightrider we need more of the technique from Parts~I--IV; thus we continue the exposition from Section \ref{essentials}.  The board is the square board $[0,1]^2$. 

The \emph{intersection lattice} $\cL(\cA_\pP)$\label{d:L} is the lattice of all intersections of subsets of $\cA_\pP$, ordered by reverse inclusion.  
The M\"obius function of $\cL(\cA_\pP)$ is denoted by $\mu$\label{d:mu}  and the bottom element is $\hat0=\bbR^{2q}$\label{d:hat0}.  
Each subspace $\cU\in\cL(\cA_\pP)$\label{d:U} is the intersection of hyperplanes involving a set $I$ consisting of $\kappa$\label{d:kappa} of the $q$ pieces.  
The \emph{essential part} of $\cU$ is the subspace $\tcU$\label{d:tU} of $\bbR^{2\kappa}$ that satisfies the same move equations as $\cU$.  
Then $\alpha(\cU;n)$,\label{d:alphaU} defined as the number of integral points in the dilation of $(\cBo)^{\kappa}\cap\tcU$, is independent of $q$ because $\tcU$ depends only on $\kappa$.  By Ehrhart theory $\alpha(\cU;n)$ is a quasipolynomial of degree $2\kappa-\codim\tcU$.  Since $\cU \cong \bbR^{2(q-\kappa)} \times \tcU$, $\codim\tcU=\codim\cU$ and the number of lattice points in $\cU \cap \ocube$ is $n^{2(q-\kappa)}\alpha(\cU;n)$.  
To get $o_\pP(q;n)$ we use M\"obius inversion to combine these numbers:
$$
o_\pP(q;n) = \sum_{\cU \in \cL(\cA_\pP)} \mu(\hat0,\cU) \, n^{2(q-\kappa)}\, \alpha(\cU;n) .
$$
This is justified by the Ehrhart theory of inside-out polytopes because the nonattacking configurations are the integral lattice points in $(n+1)\ocube$ and not in any of the hyperplanes $\cH^{d/c}_{ij}$.  

In Part~II we defined
$
\alpha^{d/c}(n) := \alpha(\cH_{12}^{d/c};n),
\label{d:adc}
$
the number of ordered pairs of positions that attack each other along slope $d/c$ (they may occupy the same position; that is considered attacking).  Similarly,  
$
\beta^{d/c}(n) := \alpha(\cH_{12}^{d/c}\cap\cH_{13}^{d/c};n),
\label{d:bdc}
$
the number of ordered triples that are collinear along slope $d/c$.
Proposition~II.3.1 gives general formulas for $\alpha$ and $\beta$.
We need only two examples:
\begin{equation}
\begin{aligned}
\alpha^{1/2}(n) 
&=\begin{cases}
\frac{5}{12}n^3+\frac{1}{3}n 	& \text{for $n$ even,} \\[3pt]
\frac{5}{12}n^3+\frac{7}{12}n 	& \text{for $n$ odd} 
\end{cases} \\&
= \frac{5}{12}n^3+\frac{11}{24}n - (-1)^n \left\{ \frac{1}{8}n \right\} ,
\\
\beta^{1/2}(n) 
&=\begin{cases}
\frac{3}{16}n^4+\frac{1}{4}n^2	& \text{for $n$ even,} \\[3pt]
\frac{3}{16}n^4+\frac{5}{8}n^2+\frac{3}{16}	& \text{for $n$ odd.} 
\end{cases} \\&
= \frac{3}{16}n^4+\frac{7}{16}n^2+\frac{3}{32} - (-1)^n\left\{ \frac{3}{16}n^2+\frac{3}{32} \right\}.
\end{aligned}
\label{E:attacklines}
\end{equation}

\subsection{Strong Parity Theorem}\label{sec:parity}\

The Parity Theorem (Theorem~II.4.1) tells us that $\alpha(\cU;n)$ is an even or odd function of $n$, depending on the codimension of $\cU$.  What it does not say is how that affects the number of undetermined coefficients in computing $\alpha(\cU;n)$, which is, in particular, the number of values of the function we need to interpolate all the coefficients.  In general, an Ehrhart quasipolynomial of degree $d$ with period $p$ has $pd+1$ coefficients that have to be computed.  (The leading coefficient is the volume of $\cU\cap\cube$ for all constituents.)  The full theorem, then, should be this:

\begin{thm}[Strong Parity Theorem]\label{T:strongparity}
For a subspace $\cU \in \cL(\cA_\pP)$ whose equations involve $\kappa$ pieces, for which $\alpha(\cU;n)$ has period $p$, the number of values of $\alpha(\cU;n)$ that are sufficient to determine all the coefficients in all constituents is $\lceil p(\kappa-\frac12\codim\cU) \rceil+\epsilon$, where $\epsilon=1$ if $\codim\cU$ is even and $0$ if it is odd.
\end{thm}

\begin{proof}
Let $\alpha(n):=\alpha(\cU;n)$ and $\nu:=\codim\cU$.  Thus, $\alpha$ has degree $d:=2\kappa-\nu$.

Let the constituents of $\alpha$ be $\alpha_0,\alpha_1,\ldots,\alpha_{p-1}$; that means $\alpha(n)=\alpha_{n\mod p}(n)$.  We take subscripts modulo $p$ so that, e.g., $\alpha_{-1}=\alpha_{p-1}$.  Write $\alpha_i(n) = a_d n^d + a_{i,d-1}n^{d-1} + \cdots + a_{i,0} n^0.$  Since 
$\alpha_{-i}(-n) = (-1)^d \alpha_i(n)$ (Corollary~II.4.1),
\begin{align*}
\alpha_{-i}(-n) &= a_d n^d + a_{-i,d-1}n^{d-1} + \cdots + a_{-i,0} n^0 = \\
(-1)^d \alpha_i(n) &= a_d n^d(-1)^0 + a_{i,d-1}n^{d-1}(-1)^1 + \cdots + a_{i,0} n^0(-1)^d.
\end{align*}
Subtracting, 
$
\sum_{j=0}^{d-1} [a_{i,j}(-1)^{d-j} - a_{-i,j}] n^j = 0,
$
which implies that
$
a_{-i,j} = (-1)^{d-j}a_{i,j}
$
for $j<d$.  It follows that only the coefficients for $0 \leq i \leq p/2$ need to be computed.  There are $d\lfloor (p-1)/2 \rfloor$ coefficients with $j<d$ for $0<i<p/2$.  For $i=0$, Corollary~II.4.1 says that $\alpha_0$ is an even or odd polynomial (depending on $d$) and so is $\alpha_{p/2}$ if the period is even.  The number of coefficients to determine, other than $\alpha_d$, is therefore $\lfloor d/2 \rfloor$ for $\alpha_0$ and the same for $\alpha_{p/2}$ if it exists.  Summing these up, there are
$$
\frac{pd}{2} + \begin{cases}
1	&\text{ if $d$ is even},\\
0	&\text{ if $d$ is odd and $p$ is even},\\
\frac12	&\text{ if $pd$ is odd}
\end{cases}
$$
independent coefficients to be computed in $\alpha$. 
\end{proof}

\section{The Nightrider }\label{N}\

The basic move set is $\M_\pN = \{(1,2),(2,1),(1,-2),(2,-1)\}$.  
As always, $u_\pN(1;n) = n^2$. It is easy to see that $u_\pN(2;1) = 0$, $u_\pN(2;2) = 6$, $u_\pN(2;3) = 28$, and not quite so easy to find $u_\pN(2;4) = 96$ by hand.  Many more values of $u_\pN(2;n)$ are in the OEIS (see Table~\ref{Tb:oeis}).  
In Theorem~II.3.1, all $(\hatc,\hatd)=(1,2)$ and the period is 2, so $\barn := (n\mod2)\in\{0,1\}$.  Therefore, 
\begin{align}\label{E:2N}
u_\pN(2;n) 
&= \left\{\frac{n^4}{2}-\frac{5 n^3}{6}+\frac{3 n^2}{2}-\frac{11 n}{12}\right\} + (-1)^n \frac{n}{4}.
\end{align}
This formula was found independently by \Kot.  His \cite[6th ed., p.\ 312]{ChMath} has an enormous formula for three nightriders (undoubtedly correct, though unproved) that is too complicated to reproduce here.  A proof may be accessible using our techniques.

We summarize the known numerical results for nightriders in Table~\ref{Tb:nightriders}.  We calculated the denominator for four nightriders by using Mathematica to find all vertices of the inside-out polytope and then the least common multiple of their denominators.

\begin{table}[htbp]
\begin{center}
\begin{tabular}{|c||c|c|c|c|}
\hline
\vstrut{12pt}   	& Types & Period & Denom  \\
\hline\hline
\vstrut{12pt} $q =$ 1		& 1	& 1 	& 1	 \\
		\hline
\vstrut{12pt} 2		& 4	& 2	& 2	 \\
		\hline
\vstrut{12pt} 3		& 36*	& 60*	& 60	 \\
		\hline
\vstrut{12pt} 4	& ---	& ---	&14559745200	 \\
		\hline
\end{tabular}
\bigskip
\end{center}
\caption{The number of combinatorial types of nonattacking placements of $q$ (unlabelled) nightriders in an $n \times n$ square board; also the period and denominator.  Note that 36 matches Conjecture III.4.4.
{\small 
\newline\hspace*{1em}* is a number derived from an empirical formula in \cite{ChMath}.  
}}
\label{Tb:nightriders}
\end{table}
\newcommand{\cha}[1]{{\bf \color{teal} Chris: #1 }}

From Theorem~II.3.1 we get a generalization.  Define a \emph{partial nightrider} $\pN^k$\label{d:partN} to have any $k$ ($1\leq k\leq4$) of the complete nightrider's moves. Up to symmetry there are five partial nightriders:
\begin{enumerate}[\qquad]
\item The {\em one-move partial nightrider} $\pN^1$ moves along slope $1/2$. \par In Corollary II.6.1 we found formulas for $u_{\pN^1}(q;n)$ with $q\leq4$.  By Theorem IV.3.2, $u_{\pN^1}(q;n)$ has period 2 for every $q\geq2$.
\item The {\em lateral nightrider} $\pN^2_{\textup{lat}}$ moves along slope $\pm1/2$.
\item The {\em inclined nightrider} $\pN^2_{\textup{incl}}$ moves along slopes $1/2$ and $2$.
\item The {\em orthogonal nightrider} $\pN^2_{\textup{orth}}$ moves along the orthogonal slopes $1/2$ and $-2$.
\item The {\em three-move partial nightrider} $\pN^3$ moves along slope $\pm1/2$ and $2$.
\end{enumerate}

\begin{cor}[of Theorem II.3.1]\label{C:2Nk}
Let $\barn:=n\mod2\in\{0,1\}$.  Then
\begin{equation}
\begin{aligned}
u_{\pN^k}(2;n) 
&= \frac{1}{2}n^4 - \frac{5k}{24}n^3 + \frac{k-1}{2}n^2 - \frac{k}{6} n - \frac{k\barn}{8} n   \\[6pt]
&= \left\{ \frac{1}{2}n^4 - \frac{5k}{24}n^3 + \frac{k-1}{2}n^2 - \frac{11k}{48}n \right\} + (-1)^n \frac{3k}{48}n .  
\end{aligned}
\label{E:u2Nk}
\end{equation}
\end{cor}

Note that all three two-move partial nightriders have the same formula.
 
A direct consequence of Theorem~II.5.1 is that we know the second coefficient of the counting quasipolynomial of $\pN^k$: 
\begin{equation}
\gamma_{1} = -\, \dfrac{5k}{24(q-2)!}.
\end{equation}
This formula for $\pN=\pN^4$ was conjectured by \Kot.  Theorem~II.5.1 gives the leading coefficient of every $\gamma_i$ for all partial nightriders.

\begin{thm}\label{T:Ngammaleading}
For a partial nightrider $\pN^k$, the coefficient $q!\gamma_i$ of $n^{2q-i}$ in $o_{\pN^k}(q;n)$ is a polynomial in $q$, periodic in $n$, with leading term 
$$
\bigg[-\frac{5k}{24}\bigg]^i \, \frac{q^{2i}}{i!}.
$$
\end{thm}

For the next result, computer algebra gave us a general formula for the third coefficient, $\gamma_{2}$.  It and the periods of $\gamma_{3}$ and $\gamma_{4}$ are new.  Both agree with \Kot's data for $\pN$ ($q\leq3$) and $\pN^1$ ($q\leq8$). 

\begin{thm}\label{thm:N}
The third coefficient of the partial nightrider counting quasipolynomial is independent of $n$; it is
$$
\gamma_{2} = \frac{1}{2!(q-2)!} \bigg\{ \bigg( \frac{5k}{24} \bigg)^2 (q-2)_2 + \bigg(\frac{k}{8}+\frac{11\lambda_1}{32}+\frac{53\lambda_2}{144}+\frac{27\lambda_3}{80}\bigg) (q-2) + (k-1) \bigg\}.
$$
The next coefficient, $\gamma_{3}$, is periodic in $n$ with period $2$ and periodic part 
\[
(-1)^n \frac{1}{3!(q-2)!} \bigg\{ \frac{3k}{8} \bigg\}.
\]
The coefficient $\gamma_{4}$ has period $2$ and periodic part 
\[
-(-1)^n \frac{1}{4!(q-2)!} \bigg\{ \frac{5k^2}{16}(q-2)_2+\bigg(\frac{3k}{2}+\frac{9\lambda_2}{4}\bigg)(q-2) \bigg\},
\]
where $\lambda_1$, $\lambda_2$, and $\lambda_3$ depend on the partial nightrider as stated in Table~\ref{Tb:lambda}.
\end{thm}

\begin{table}[hb]
$\begin{array}{l|cccccc}
 & \pN^1 &\pN^2_{\textup{lat}} & \pN^2_{\textup{incl}} & \pN^2_{\textup{orth}}  & \pN^3 & \pN^4 \\[3pt] \hline
\lambda_1 & 0 & 1 & 0 & 0 & 1 & 2 \\ 
\lambda_2 & 0 & 0 & 1 & 0 & 1 & 2 \\ 
\lambda_3 & 0 & 0 & 0 & 1 & 1 & 2 \\
\end{array}$
\bigskip
\caption{The values of the $\lambda_i$.\label{d:lambda}}
\label{Tb:lambda}
\end{table}

These formulas with $k=1$ agree with the ones in Corollary II.6.1, where they were presented as a simple case of general one-move rider formulas.

\begin{proof}
Just as in Theorem~\Phvdiag, we calculate $\gamma_{2}$ by determining the contribution from all subspaces $\cU$ defined by two move equations, each of the form $\cH_{ij}^{d/c}$ for $d/c\in\{1/2,2/1,-1/2,-2/1\}$ and $i,j\in[q]$.  This is done in Lemma~\ref{L:Ncodim2}.  
  
There is no contribution to $\gamma_{2}$ from subspaces of codimension $0$ or $1$.

The coefficient $\gamma_{3}$ may have contributions from subspaces of codimensions $1$ to $3$.  Since the contribution of a subspace of codimension $3$ comes from the leading coefficient, it is independent of $n$.  We did not compute these leading coefficients.  
A subspace of codimension $2$ contributes zero by Theorem~II.4.2, or simply by observing the formula in Equation~(II.2.5).  
Therefore the periodic part of $(q-2)!\gamma_{3}$ arises solely from $\alpha(\cH^{1/2};n)=\alpha^{1/2}(n)$ in Equation~\eqref{E:attacklines}, which provides a periodic contribution of $(-1)^n\frac{1}{8}n^{2q-3}\binom{q}{2}$ from each of the $k$ hyperplanes in $\cA_{\pN^k}$.

The calculations for hyperplanes and subspaces of Type $\cU_{3\mathrm{b}}^2$ in Lemma~\ref{L:Ncodim2} imply that $\gamma_{4}$ is periodic with period $2$ when $k\geq 2$.  That is because a periodic contribution can come only from a subspace of codimension 1, 2, or 3.  Equation~\eqref{E:attacklines} shows that hyperplanes make no contribution to  $\gamma_{4}$.  We can therefore read the periodic contribution of $- (-1)^n\big[ \frac{k}{16}(q)_3+\lambda_2\frac{3}{32}(q)_3+\frac{5k^2}{384}(q)_4 \big]$ from Lemma~\ref{L:Ncodim2}.
\end{proof}

\begin{lem}\label{L:Ncodim2}
The total contribution to $o_{\pN^k}(q;n)$ of all subspaces with codimension $2$ is 
\begin{align*}
&\left\{ \left[ \frac{k-1}{2}(q)_2+\frac{k}{16}(q)_3+\left(\lambda_1\frac{11}{64}+\lambda_2\frac{53}{288}+\lambda_3\frac{27}{160}\right)(q)_3+\frac{25k^2}{1152}(q)_4 \right] n^{2q-2}  \right.
\\&\quad
\left. + \left[ \frac{7k}{48}(q)_3+\left(\lambda_1\frac{k}{4}+\lambda_2\frac{83}{288}+\lambda_3\frac{1}{4}\right)(q)_3+\frac{55k^2}{1152}(q)_4 \right] n^{2q-4}  \right. 
\\&\quad
\left. + \left[ \frac{1}{32}(q)_3+\frac{65}{144}(q)_4 \right] n^{2q-6} \right\} 
\\&
- (-1)^n \left\{ \left[ \frac{k}{16}(q)_3+\lambda_2\frac{3}{32}(q)_3+\frac{5k^2}{384}(q)_4 \right] n^{2q-4} + \left[ \frac{k}{32}(q)_3+\frac{11k^2}{768}(q)_4 \right] n^{2q-6} \right\}
\\&
+ \bigg[\lambda_1\bigg(\frac{51}{256}-\frac{19}{256}\zeta_4^n-\frac{13}{256} \zeta_4^{2 n}-\frac{19}{256} \zeta_4^{3 n}\bigg)
\\&
\ \ \quad+
\lambda_2\bigg(
\frac{527}{1728}
-\frac{1}{8} \zeta_{12}^{3n}
+\frac{2}{27} \zeta_{12}^{4n} 
-\frac{13}{64} \zeta_{12}^{6n}
+\frac{2}{27} \zeta_{12}^{8n}
-\frac{1}{8} \zeta_{12}^{9n}\bigg)
\\&
\ \ \quad
+\lambda_3\bigg(\frac{599}{1600}
-\frac{4}{25} \zeta_{20}^{4 n}
+\frac{1}{8}  \zeta_{20}^{5 n}
-\frac{4}{25} \zeta_{20}^{8 n} 
+\frac{1}{64} \zeta_{20}^{10 n}
-\frac{4}{25} \zeta_{20}^{12 n} 
+\frac{1}{8}  \zeta_{20}^{15 n}
-\frac{4}{25} \zeta_{20}^{16 n}\bigg)\bigg] (q)_3\,n^{2q-6},
\end{align*}
\label{E:Ncodim2}
where $\lambda_1$, $\lambda_2$, and $\lambda_3$ are as in Table~\ref{Tb:lambda}.
\end{lem}

The symmetry in the coefficients of powers of each $\zeta_r$ in the last $n^{2q-6}$-term is due to the coefficients' being real numbers.

\begin{proof} 
There are four types of subspace, of which only Type $\cU_{3\mathrm{b}}^2$ involves calculations that are substantially different from those in Lemma~\Lcodimzot.  That is where the $\lambda_i$ arise.

From Section~\ref{Nprep}, for a subspace determined by equations involving $\kappa$ pieces, $\alpha(\cU;n)$ is a quasipolynomial of degree $2\kappa-\codim\cU$.

\medskip
\begin{description}
\item[{\bf Type $\cU_2^2$}\,] The subspace $\cU$ is defined by two move equations involving the same two pieces.  There is one such subspace for each of the $\binom{q}{2}$ unordered pairs of pieces and $n^2$ ways to place the attacking pieces in $\cU$.  Since $\cU$ lies in $k$ hyperplanes, the M\"obius function is $\mu(\hat0,\cU)  = k-1$. The contribution to $o_{\pN^k}(q;n)$ is $\frac{k-1}{2} (q)_2 n^{2q-2}$ so the contribution to $q!\gamma_{2}$ is $\frac{k-1}{2}(q)_2$.

The one-move partial nightrider is an exception, since there is no subspace of this type when $k=1$.  Then instead of multiplying the contribution by $\mu(\hat0,\cU)$, we multiply it by $0=k-1$; i.e., the same multiplier expression.

\medskip
\item[{\bf Type $\cU_{3\mathrm{a}}^2$}\,]  The subspace $\cU$ is defined by two move equations of the same slope involving three pieces. There is one subspace of this type for each of the $k$ slopes.  The number of points in each subspace is $\beta^{1/2}(n)$ from Equation~\eqref{E:attacklines}.  There are $(q)_3/3!$ ways to choose three partial nightriders, and the M\"obius function is $2$.  Thus we multiply $\beta^{1/2}(n)$ by $\frac{2k}{3!}(q)_3n^{2q-6}$ to find that the contribution to $o_{\pN^k}(q;n)$ is 
$$
(q)_3 \left\{ \frac{k}{16}n^{2q-2} + \frac{7k}{48}n^{2q-4} + \frac{k}{32}n^{2q-6} - (-1)^n \Big[ \frac{k}{16}n^{2q-4} + \frac{k}{32}n^{2q-6} \Big] \right\} ,
$$
so that the contribution to $q!\gamma_{2}$ is $\frac{k}{16}(q)_3$ and that to $q!\gamma_{4}$ is $\big[\frac{7k}{48} - (-1)^n\frac{k}{16}\big](q)_3$.

\medskip
\item[{\bf Type $\cU_{3\mathrm{b}}^2$}\,]  
The subspace $\cU$ is defined by two move equations of different slopes involving three pieces, say $\cU=\cH^{d/c}_{12} \cap \cH^{d'/c'}_{23}$.  By symmetry it suffices to find the contributions when $d/c=1/2$ and $d'/c'\in\{2/1,-2/1,-1/2\}$.  The total contribution will depend on the piece's basic move set.  We write $\cU=\cU^{d'/c'}$ when we need to mention the slope. The M\"obius function is $\mu(\hat0,\cU)  = 1$.

We choose the piece $\pP_2$ in $q$ ways, $\pP_1$ in $q-1$ ways, and $\pP_3$ in $q-2$ ways.

For each value of $d'/c'$ we calculated the denominators of all vertex coordinates of $\cube \cap \cU$ using Mathematica.  That gave us the denominator $D(\cube \cap \cU)$ and hence an upper bound on the period of $\alpha(\cU;n)$ in each case.  
Using Mathematica again we found quasipolynomial formulas for the number of placements of the three nightriders, $\alpha(\cU;n)$.  These formulas were calculated by varying the position of $\pP_2$ in the $n\times n$ grid as $n$ varied in a residue class modulo $D(\cube \cap \cU)$.  The calculations were carried out for $n=1,\ldots,100$, which covers at least five periods in every case.  By Theorem~\ref{T:strongparity} and the fact that $\alpha(\cU;n)$ has degree $4 = 2\cdot3-\codim\cU$, there are $2p+1$ coefficients to determine in $\alpha(\cU;n)$; as the periods are bounded by 20 in every case, that is enough data to infer them all with redundancy.  
(We found that the period of every one of these quasipolynomials equals the denominator of the subspace, which tends to support Conjecture~\ref{Cj:p=D}.)

{\em Case $d'/c'=-1/2$.} The vertex denominators here are $2$ and $4$ so $D(\cube \cap \cU)=4$.   The number of placements is
\[
\alpha(\cU^{-1/2};n) = 
\begin{cases}
\frac{11}{64}n^4+\frac{1}{4}n^2 & \text{for $n\equiv 0\mod4$,} \\[3pt]
\frac{11}{64}n^4+\frac{1}{4}n^2+\frac{1}{4} & \text{for $n\equiv 2\mod4$,} \\[3pt]
\frac{11}{64}n^4+\frac{1}{4}n^2+\frac{19}{64} & \text{for $n$ odd.} 
\end{cases}
\]

{\em Case $d'/c'=2/1$.}  The vertex denominators here are $2,3,4$ so $D(\cube \cap \cU) = \lcm(2,3,4) = 12$.   The number of placements is
\[
\alpha(\cU^{1/2};n) = 
\begin{cases}
\frac{53}{288}n^4+\frac{7}{36}n^2 & \text{for $n\equiv 0\mod 12$,} \\[3pt]
\frac{53}{288}n^4+\frac{7}{36}n^2-\frac{2}{9} & \text{for $n\equiv \pm4\mod 12$,} \\[3pt]
\frac{53}{288}n^4+\frac{7}{36}n^2+\frac{1}{2} & \text{for $n\equiv 6\mod 12$,} \\[3pt]
\frac{53}{288}n^4+\frac{7}{36}n^2+\frac{5}{18} & \text{for $n\equiv \pm2\mod 12$,} \\[3pt]
\frac{53}{288}n^4+\frac{55}{144}n^2+\frac{21}{32} & \text{for $n\equiv 3\mod 6$,} \\[3pt]
\frac{53}{288}n^4+\frac{55}{144}n^2+\frac{125}{288} & \text{for $n\equiv \pm1\mod 6$.} 
\end{cases}
\]
Note that the coefficient of $n^2$, which becomes a contribution to $\gamma_{4}$, has period $2$.

{\em Case $d'/c'=-2/1$.}  The vertex coordinate denominators here are $2$, $4$, and $5$ so $D(\cube \cap \cU) = 20$.  The number of placements is
\[\alpha(\cU^{-2/1};n) = 
\begin{cases}
\frac{27}{160}n^4+\frac{1}{4}n^2 & \text{for $n\equiv 0\mod 20$,} \\[3pt]
\frac{27}{160}n^4+\frac{1}{4}n^2+\frac{4}{5} & \text{for $n\equiv \pm4,\pm8 \mod 20$,} \\[3pt]
\frac{27}{160}n^4+\frac{1}{4}n^2-\frac{1}{2} & \text{for $n\equiv 10\mod 20$,} \\[3pt]
\frac{27}{160}n^4+\frac{1}{4}n^2+\frac{3}{10} & \text{for $n\equiv \pm2, \pm6 \mod 20$,} \\[3pt]
\frac{27}{160}n^4+\frac{1}{4}n^2-\frac{9}{32} & \text{for $n\equiv \pm5\mod 20$,} \\[3pt]
\frac{27}{160}n^4+\frac{1}{4}n^2+\frac{83}{160} & \text{for odd $n\not\equiv \pm5\mod 20$.} 
\end{cases}\]

For the complete nightrider, the $\binom{4}{2}$ choices of pairs of slopes consist of two copies of each case.  For the three-move partial nightrider, the three choices of pairs of slopes consist of one copy of each case.  And for the two-move partial nightriders, there is only the one case corresponding to its pair of slopes.  These conditions are encoded by the coefficients $\lambda_i$.  Therefore the contribution to $o_{\pN^k}(q;n)$ is $(q)_3$ times 
\begin{align*}
&\quad \ 
\lambda_1 \bigg\{\frac{11}{64} n^{2q-2}+ \frac{1}{4}n^{2q-4}+ \bigg(
\frac{51}{256}
-\frac{19}{256}\zeta_4^n
-\frac{13}{256} \zeta_4^{2 n}
-\frac{19}{256} \zeta_4^{3 n}
\bigg)n^{2q-6} \bigg\}
\\&
+\lambda_2\bigg\{\frac{53}{288} n^{2q-2}+ \bigg(\frac{83}{288}-(-1)^n\frac{3}{32} \bigg)n^{2q-4}
\\& \qquad \quad
+ \bigg(
\frac{527}{1728}
-\frac{1}{8} \zeta_{12}^{3n}
+\frac{2}{27} \zeta_{12}^{4n} 
-\frac{13}{64} \zeta_{12}^{6n}
+\frac{2}{27} \zeta_{12}^{8n}
-\frac{1}{8} \zeta_{12}^{9n}
\bigg)n^{2q-6} \bigg\} 
\\&
+\lambda_3 \bigg\{\frac{27}{160} n^{2q-2}+ \frac{1}{4}n^{2q-4}
\\&\qquad\quad 
+ \bigg(
\frac{599}{1600}
-\frac{4}{25} \zeta_{20}^{4 n}
+\frac{1}{8}  \zeta_{20}^{5 n}
-\frac{4}{25} \zeta_{20}^{8 n} 
+\frac{1}{64} \zeta_{20}^{10 n}
-\frac{4}{25} \zeta_{20}^{12 n} 
+\frac{1}{8}  \zeta_{20}^{15 n}
-\frac{4}{25} \zeta_{20}^{16 n}
\bigg)n^{2q-6} \bigg\}
\\\end{align*}

\medskip
\item[{\bf Type $\cU_{4^*}^2{:}\cU_2^1\cU_2^1$\,}] The subspace $\cU$ is defined by two move equations involving four distinct pieces.
For every pair of hyperplanes, the number of attacking configurations is $\big(\alpha^{1/2}(n)\big)^2$, whose value is given in Equation~\eqref{E:attacklines}.  
We must also multiply by the number of ways in which we can choose this pair of hyperplanes, which is $k\frac{(q)_4}{8}+\binom{k}{2}\frac{(q)_4}{4}=\frac{k^2}{8}(q)_4$.  The M\"obius function is $1$.  We conclude that the contribution to $o_\pN(q;n)$ is 
\[
(q)_4 \left\{ \frac{25k^2}{1152}n^{2q-2} + \frac{55k^2}{1152}n^{2q-4} + \frac{65k^2}{2304}n^{2q-6} 
- (-1)^n\bigg[\frac{5k^2}{384} n^{2q-4} + \frac{11k^2}{768} n^{2q-6}\bigg] \right\},
\]  
that to $\gamma_{2}$ is $\frac{25k^2}{1152}(q)_4$, and that to $\gamma_{4}$ is $\big( \frac{55k^2}{1152}-(-1)^n\frac{5k^2}{384} \big)(q)_4$.
\end{description}
\end{proof}

Curiously, not only is the quasipolynomial for every subspace as a whole an even function, so is each constituent; equivalently, opposite constituents $\alpha_i(\cU;n)$ and $\alpha_{-i}(\cU;n)$ are equal, for every $i$.  We do not know why.

Type $\cU_{3\mathrm{b}}^2$ contributes period  $60 = \lcm(12,20,4)$ to $\gamma_{6}$ for the complete nightrider and the three-move partial nightrider, as one can see from Lemma~\ref{L:Ncodim2}.  We therefore expect $\gamma_{6}$ to have period that is a multiple of 60 for those pieces; however, we are far from proving this.

This computational method can be applied to larger numbers of any piece, limited only by human effort and computing power.  It should be feasible to deduce, at the least, partial nightrider formulas for $\gamma_{3}$, $\gamma_{4}$, and $u_{\pN^k}(3;n)$.


\section{Conclusions, Conjectures, Extensions}\label{last}

Work on nonattacking chess placements raises many questions, some of which have general interest.


\subsection{Simplified riders }\label{simplified}\

We cannot reach satisfactorily strong conclusions about the queen and nightrider in part because their periods grow too rapidly as $q$ increases, which we now understand by way of the twisted Fibonacci spirals in Section~\fourmove.  It would be desirable to study simplified analogs, hoping not only for hints to solve those pieces but to find general patterns in the period and coefficients.  As having four move directions is complicated, we propose handicapping the pieces by eliminating some of their moves.

As we saw in Part~III, partial queens $\pQ^{hk}$ are approachable because the queen's moves are individually simple.  We suggest further study of the following variants, some of which have been investigated by \Kot. 

\begin{enumerate}[(a)]
\item A generalization of the anassa is a rider with two moves, $(1,0)$ and $(c,d)$.  The denominator of this piece was determined in Proposition~\denomtwomove.  This piece, and especially its period, would facilitate analysis of the effect of non-unit slopes.
\end{enumerate}
The nightrider's main complication comes from the non-unit slopes.  We propose as worthy subjects the partial nightriders with only two moves, from Section~\ref{N}:
\begin{enumerate}[(a)]
\item[(b)] The lateral nightrider.  We conjecture a period of $4$ for $q\geq3$.  We verified this for $q=3,4$.
\item[(c)] The inclined nightrider.  The period for $q=3$ is 12.
\item[(d)] The orthogonal nightrider.  The period for $q=3$ is 20.
\end{enumerate}
We thank Arvind Mahankali for calculating these periods, using Mathematica.  Note that the coefficient of $n^{2q-6}$ in Lemma~\ref{L:Ncodim2} strongly suggests but does not prove these periods.

After we proposed these pieces, Hanusa and Mahankali \cite{Arvind} found their denominators, which we reproduce in Table~\ref{Tb:N2D}.  As upper bounds on periods, they also suggest but do not prove the true periods.  
\begin{table}[hb]
\begin{tabular}{l|cccc}
\quad\emph{Name}	&$q=1$	&$q=2$	&$q=3$	&$q\geq4$	\\[1pt]
\hline
Lateral	&1	&2	&4	&4	\\[1pt]
Inclined	&1	&2	&$12=3\cdot2^{q-1}$	&$3\cdot2^{q-1}$	\\ [1pt]
Orthogonal	&1	&2	&$20=5\cdot2^{q-1}$	&$15\cdot 2^{q-1}$	\\[1pt]
\end{tabular}
\medskip
\caption{The denominators of the two-move partial nightriders, from \cite[Section 6]{Arvind}.}
\label{Tb:N2D}
\end{table}

\begin{enumerate}[(a)]
\item[(e)]  A simple three-move rider would have moves $(1,0), (0,1), (c,d)$.  This should be investigated.
\item[(f)]  The partial nightrider $\pN^3$.  We discussed it briefly in Section~\ref{N}, finding the counting formula and the period $2$ for $q=2$.  The period for three pieces appears to be $60$.  This period is the same as for the complete nightrider but we expect $\pN^3$ to have a smaller period than $\pN$ when $q\geq4$.
\end{enumerate}
%


\subsection{Counting nonattacking combinatorial types}\label{configs}\

It would be valuable to produce a conjectural expression for the number of combinatorial types of nonattacking configuration for the queen, a partial queen with three moves, or any other piece with more than two moves (one or two moves being easy; see Proposition~I.5.6).


\subsection{The problem of the recurrence relation}\label{rr}\

Suppose we generate a sequence of numbers, say $u_1,u_2,\ldots$, one at a time.  Each time we generate a new $u_n$ we look for a linear, homogeneous recurrence relation with constant coefficients that $u_1,u_2,\ldots,u_n$ satisfy.  The recurrence may change at each $n$, but suppose there is an $N$ such that the $(N+k)$-th recurrence is identical to the $N$-th recurrence for $k = 1,2,\ldots, K$.  Let us call this a \emph{stable recurrence} of strength $K$.  If the stable recurrence is valid for all $n>N$, it is the true recurrence satisfied by the sequence.

\Kot's method of finding formulas for a rider $\pP$ and a number $q$ was to compute $u_n=u_\pP(q;n)$ until he detected a stable recurrence of sufficient strength.  Then he assumed the stable recurrence was the true recurrence and used it to obtain a rational generating function for $u_\pP(q;n)$, from which he could obtain a quasipolynomial formula.  (Personal communication, 19 July 2018.)
His stable relation is the true relation in all cases that we could verify.  The first question is: Why?

We know there is a recurrence for $u_\pP(q;n)$, because that is implied by the quasipolynomial formula.  What we do not know is whether there is an $N$ at which a stable recurrence appears which is not the true recurrence.  That means the $N$-th through $(N+K)$-th recurrences are the same, for some positive $K$, but the $(N+K+1)$-st recurrence is not.  We conjecture that this cannot happen for (a) some knowable value of $K$; (b) $K=1$.

\begin{conj}\label{Cj:recurrence}
(a) For any rider $\pP$ and fixed $q>0$, the sequence $u_\pP(q;n)$ satisfies a recurrence relation which is the first stable relation of sufficient strength found as the numbers $u_\pP(q;n)$ are generated.

(b) The first repeated relation is the true relation.
\end{conj}

As we mentioned in Section~\ref{queen}, the recurrences found by \Kot\ are much shorter than the quasipolynomial period.  That cannot be an accident!  It leads to the second important problem about recurrences.

\begin{conj}
All sequences $\{u_\pP(q;n)\}_n$ satisfy recurrences that are much shorter than the period of the quasipolynomial.  
\end{conj}

Supposing they do, we ask how much shorter, and why?  It is due to factors cancelling in the Ehrhart rational generating function, whose denominator is $(1-t^p)^{2q+1}$.  Which factors cancel, and why do they exist in the numerator?

A simple example: The Ehrhart denominator for $q$ bishops is $(1-t^2)^{2q+1}$; but \Kot\ finds that it cancels down to $(1-t^2)^{2q-5} (1+t)^{6}$ \cite[p.\ 239]{ChMath}.


\section*{Dictionary of Notation}

\begin{enumerate}[]
\item $(c,d)$ -- coordinates of move vector (p.\ \pageref{d:mr})
\item $d/c$ -- slope of line or move (p.\ \pageref{d:slope-hyp})
\item $F_q$ -- Fibonacci numbers (p.\ \pageref{d:Fib})
\item $h$ -- \# of horizontal, vertical moves of partial queen (p.\ \pageref{d:partQ})
\item $k$ -- \# of diagonal moves of partial queen (p.\ \pageref{d:partQ})
\item $k$ -- \# of moves of partial nightrider (p.\ \pageref{d:partN})
\item $m = (c,d)$ -- basic move (p.\ \pageref{d:mr})
\item $m^\perp = (d,-c)$ -- orthogonal vector to move $m$ (p.\ \pageref{d:mrperp})
\item $n$  -- size of square board (p.\ \pageref{d:n})
\item $o_\pP(q;n)$ -- \# of nonattacking labelled configurations (p.\ \pageref{d:distattacks})
\item $p$ -- period of quasipolynomial (p.\ \pageref{d:p})
\item $q$ -- \# of pieces on a board (p.\ \pageref{d:q})
\item $s(n,k)$ -- Stirling number of the first kind (p.\ \pageref{d:s})
\item $S(n,k)$ -- Stirling number of the second kind (p.\ \pageref{d:S})
\item $u_\pP(q;n)$ -- \# of nonattacking unlabelled configurations (p.\ \pageref{d:indistattacks})
\item $z=(x,y)$, $z_i=(x_i,y_i)$ -- piece position (p.\ \pageref{d:config})
\item $\bz=(z_1,\ldots,z_q)$ -- configuration (p.\ \pageref{d:config})
\end{enumerate}
\smallskip
\begin{enumerate}[]
\item $\alpha(\cU;n)$ -- \# of attacking configurations in essential part of subspace $\cU$ (p.\ \pageref{d:alphaU})
\item $\alpha^{d/c}(n)$ -- \# of 2-piece attacks on slope $d/c$ (p.\ \pageref{d:adc})
\item $\beta^{d/c}(n)$ -- \# of 3-piece attacks on slope $d/c$ (p.\ \pageref{d:bdc})
\item $\gamma_{i}$ -- coefficient of $n^{2q-i}$ in $u_\pP$ (p.\ \pageref{d:gamma})
\item $\zeta_r=e^{2\pi i/r}$ -- primitive $r$-th root of unity (p.\ \pageref{d:zetar})
\item $\kappa$ -- \# of of pieces in equations of $\cU$ (p.\ \pageref{d:kappa})
\item $\lambda_1,\lambda_3,\lambda_2$ -- coefficients for partial queens $\pN^k$ (p.\ \pageref{d:lambda})
\item $\mu$ -- M\"obius function of intersection lattice (p.\ \pageref{d:mu}) 
\end{enumerate}
\smallskip
\begin{enumerate}[]
\item $D$ -- denominator of inside-out polytope (p.\ \pageref{d:D}) 
\end{enumerate}
\smallskip
\begin{enumerate}[]
\item $\M$ -- set of basic moves (p.\ \pageref{d:moveset})
\end{enumerate}
\smallskip
\begin{enumerate}[]
\item $\cA_\pP$ -- move arrangement of piece $\pP$ (p.\ \pageref{d:AP})
\item $\cB$ -- closed board: usually the square $[0,1]^2$ (p.\ \pageref{d:cB})
\item $\cH_{ij}^{d/c}$ -- hyperplane for move $(c,d)$ (p.\ \pageref{d:slope-hyp})
\item $\cL(\cA_\pP)$ -- intersection lattice  (p.\ \pageref{d:L})
\item $\cube, \ocube$ -- closed, open unit hypercube (p.\ \pageref{d:cP})
\item $(\cube,\cA_\pP)$ -- inside-out polytope (p.\ \pageref{d:cP})
\item $\cT$ -- triangular board, $0 \leq x \leq y \leq 1$ (p.\ \pageref{d:cT}
\item $\cU$ -- subspace in intersection lattice (p.\ \pageref{d:U}) 
\item $\tcU$ -- essential part of subspace $\cU$ (p.\ \pageref{d:tU}) 
\end{enumerate}
\smallskip
\begin{enumerate}[]
\item $\hat0= \bbR^{2q}$ -- bottom element of intersection lattice (p.\ \pageref{d:hat0})
\end{enumerate}
\smallskip
\begin{enumerate}[]
\item $\bbR$ -- real numbers
\item $\bbR^{2q}$ -- configuration space (p.\ \pageref{d:configsp}) 
\item $\bbZ$ -- integers
\end{enumerate}
\smallskip
\begin{enumerate}[]
\item $\pB$ -- bishop (p.\ \pageref{B})
\item $\pN$ -- nightrider (p.\ \pageref{N})
\item $\pN^k$ -- partial nightrider (p.\ \pageref{d:partN})
\item $\pP$ -- piece (p.\ \pageref{d:P})
\item $\pQ$ -- queen (p.\ \pageref{Q})
\item $\pQ^{hk}$ -- partial queen (p.\ \pageref{d:partQ})
\item $\pR$ -- rook (p.\ \pageref{R})
\end{enumerate}

\newpage
\newcommand\otopu{\r{u}}

\end{document}